\documentclass[12pt]{amsart}
\usepackage{amssymb,amsmath}
\def\C {{\mathcal C}}

\def\H {{\mathcal H}}

\def\M {{\mathcal M}}

\def\A {{\mathbb A}}

\def\R {\mathbb{R}}
\def\N {\mathbb{N}}

\def\D {{\mathfrak D}}

\def\II {{\mathbb S}}

\def\Re {\mathfrak{Re\,}}

\def\eps{\varepsilon}
\def\e{{\rm e}}

\def\d{{\rm d}}
\def\ddt{\frac{\d}{\d t}}

\def\ii{{\rm i}}
\def \l {\langle}
\def \r {\rangle}

\def \tx {\textstyle}
\def\uts {{\sc uts}}

\def \and {{\qquad\text{and}\qquad}}

\newtheorem{proposition}{Proposition}
\newtheorem{theorem}[proposition]{Theorem}

\newtheorem{lemma}[proposition]{Lemma}
\theoremstyle{definition}

\newtheorem*{remark}{Remark}
\numberwithin{equation}{section}
\def \au {\rm}
\def \ti {\it}
\def \jou {\rm}

\def \no#1#2#3 {{\bf #1} (#3), #2.}
\def \eds#1#2#3 {#1, #2, #3.}

\title[Timoshenko systems with fading memory]
{Timoshenko systems with fading memory}

\author[M. Conti, F. Dell'Oro and V. Pata]
{Monica Conti, Filippo Dell'Oro and Vittorino Pata}

\address{Politecnico di Milano - Dipartimento di Matematica ``F.\ Brioschi''
\newline\indent
Via Bonardi 9, 20133 Milano, Italy}
\email{monica.conti@polimi.it}
\email{filippo.delloro@mail.polimi.it}
\email{vittorino.pata@polimi.it}


\begin{document}

\begin{abstract}
The decay properties of the semigroup
generated by a linear Timoshenko system with fading memory
are discussed.
Uniform stability is shown to occur within a necessary and
sufficient condition on the memory kernel
$\mu$.
\end{abstract}

\maketitle

\section{Introduction}

\noindent
Given a real interval ${\mathfrak I}=[0,\ell]$,
we consider the viscoelastic beam model of Timoshenko type \cite{TIM}
\begin{equation}
\label{TIMO0}
\begin{cases}
\rho_1 \varphi_{tt} -\kappa(\varphi_x +\psi)_x = 0\\
\noalign{\vskip1.5mm}
\displaystyle
\rho_2 \psi_{tt} -b\psi_{xx} +\int_0^\infty \mu(s)\psi_{xx}(t-s)\,\d s+\kappa(\varphi_x +\psi)= 0
\end{cases}
\end{equation}
in the unknowns
$\varphi,\psi: \,(x,t)\in {\mathfrak I}\times [0,\infty) \mapsto \R$, where the strictly
positive constants
$\rho_1,\rho_2,\kappa,b$ satisfy the relation
\begin{equation}
\label{rivera}
\frac{\rho_1}{\kappa}=\frac{\rho_2}{b}
\end{equation}
while the memory kernel $\mu$ is a (nonnegative) nonincreasing absolutely continuous function
on $[0,\infty)$ such that
$$
a:=b-m>0\qquad
\text{where}\qquad
m:=\int_0^\infty \mu(s)\,\d s>0.
$$
The system is complemented with
the Dirichlet boundary conditions
$$
\varphi(0,t)=\varphi(\ell,t)=\psi(0,t)=\psi(\ell,t)=0,
$$
but our arguments can be used to prove analogous results for other kind of boundary
conditions as well, such as Dirichlet/Neumann or Neumann/Dirichlet.

Following \cite{MRFS}, by rephrasing the problem within the history framework of Dafermos~\cite{DAF},
system \eqref{TIMO0} is shown to generate a contraction semigroup $S(t)$ of solutions
acting on a suitable Hilbert space $\H$
accounting for the presence of the memory.
The aim of this work is to establish a necessary and sufficient condition
on the memory kernel $\mu$ (within the class of kernels considered above)
in order for $S(t)$ to be exponentially stable on $\H$, namely,
$$\|S(t)z\|_{\H}\leq K\e^{-\omega t}\|z\|_\H,\quad\forall z\in\H,$$
for some $K\geq 1$ and $\omega>0$.
Our main theorem reads as follows.

\begin{theorem}
\label{MAIN}
The semigroup $S(t)$ is exponentially stable if and only if
there exist $C\geq 1$ and $\delta>0$ such that
\begin{equation}
\label{NEC}
\mu(\sigma+s)\leq C \e^{-\delta \sigma}\mu(s)
\end{equation}
for every $\sigma\geq 0$ and $s>0$.
\end{theorem}

The decay properties of $S(t)$ have been previously studied in the papers \cite{MS,MRFS}.
We will discuss and compare those results in the next Sections \ref{Nec} and \ref{Suf},
where we will also provide the proofs of the two directions of Theorem~\ref{MAIN}.
Condition~\eqref{NEC} appears for the first time in connection with systems with memory
in~\cite{ChePat}, whereas \eqref{rivera}, which basically says that the two hyperbolic equations share the same
propagation speed, is used in~\cite{AMM} for
the same model but with a convolution integral of Volterra type. In that work, the
failure of~\eqref{rivera} is shown to prevent the possibility of any uniform decay of the solutions.
Actually, the same phenomenon pops up if the convolution integral, which contains the whole dissipation of
the system, is replaced by an instantaneous damping term, such as $\psi_t$
(see \cite{SOU}).

\begin{remark}
The existence of the semigroup $S(t)$ can be actually established under weaker
conditions on $\mu$, which can be only piecewise absolutely continuous
with (infinitely many) discontinuity points.
For the proof of the necessity part of Theorem~\ref{MAIN} nothing more is needed.
In particular, no use is made of~\eqref{rivera}.
Concerning sufficiency, besides~\eqref{rivera}-\eqref{NEC},
one has to require that the set where $\mu'<0$ has positive measure,
automatically satisfied if $\mu$ is absolutely continuous (as in our hypotheses).
\end{remark}

\section{Functional Setting and Notation}
\label{FSN}

\noindent
In what follows, $\langle\cdot,\cdot\rangle$ and $\|\cdot\|$ are the
inner product and norm on the (real) Hilbert space $L^2({\mathfrak I})$.
We will also consider the Sobolev space
$H_0^1({\mathfrak I})$ endowed with the gradient norm, due to the Poincar\'e inequality,
along with the $L^2$-weighted space of $H_0^1$-valued functions on $\R^+=(0,\infty)$
$$
\M = L^2(\R^+; H_0^1({\mathfrak I})),
\qquad
\langle\eta,\xi\rangle_\M=\int_0^\infty \mu(s)\langle\eta_x(s),\xi_x(s)\rangle\,\d s.
$$
We define the linear operator $T$ on $\M$ by
$$T\eta=-D\eta,\qquad \D(T)=\big\{\eta\in{\M}:D\eta\in\M,\,\,
\eta(0)=0\big\},$$
where $D$ stands for weak derivative.
The operator $T$ is the infinitesimal generator
of the right-translation semigroup $\Sigma(t)$
on $\M$, acting as
$$
[\Sigma(t)\eta](s)=
\begin{cases}
0 & s\leq t,\\
\eta(s-t) & s>t.
\end{cases}
$$
The phase space of our problem will be\footnote{In the case of Neumann
boundary conditions, one has to work in spaces of zero-mean functions.}
$$
\H = H_0^1({\mathfrak I})\times L^2({\mathfrak I})\times
H_0^1({\mathfrak I})\times L^2({\mathfrak I})\times \M
$$
normed by
$$
\|(\varphi,\tilde\varphi,\psi,\tilde\psi,\eta)\|_\H^2
= \kappa\|\varphi_x+\psi\|^2 +\rho_1\|\tilde\varphi\|^2
+a\|\psi_x\|^2+\rho_2\|\tilde\psi\|^2 + \|\eta\|^2_\M.
$$

\smallskip
\noindent
Finally, we recall \cite[Theorem 3.3]{ChePat}.

\begin{theorem}
\label{THM33}
The right-translation semigroup $\Sigma(t)$ acting on $\M$ is exponentially stable if and only if
\eqref{NEC} holds.
\end{theorem}

\section{The Contraction Semigroup}
\label{Semigroup}

\noindent
We formally define the auxiliary variable
$\eta=\eta^t(s)$ as (the dependence on $x$ is omitted)
$$
\eta^t(s)=\psi(t)-\psi(t-s).
$$
Then, \eqref{TIMO0} turns into the
system in the unknowns $\varphi=\varphi(t)$, $\psi=\psi(t)$ and $\eta=\eta^t$
\begin{equation}
\label{sistema}
\begin{cases}
\rho_1 \varphi_{tt} -\kappa(\varphi_x +\psi)_x = 0,\\
\noalign{\vskip1mm}
\displaystyle
\rho_2 \psi_{tt} -a\psi_{xx}-\int_0^\infty \mu(s)\eta_{xx}(s)\,\d s+\kappa(\varphi_x +\psi)= 0,\\
\noalign{\vskip1mm}
\eta_t = T\eta + \psi_t.
\end{cases}
\end{equation}
Introducing the state vector
$$Z(t)=(\varphi(t),\tilde\varphi(t),\psi(t),\tilde\psi(t),\eta^t),$$
system \eqref{sistema} can be clearly written as
a linear ODE in $\H$ of the form
\begin{equation}
\label{EQLIN}
\ddt Z(t)=\A Z(t),
\end{equation}
where the domain $\D(\A)$ of the linear operator $\A$, whose action
can be easily deduced from~\eqref{sistema},
is made by all the vectors $(\varphi,\tilde\varphi,\psi,\tilde\psi,\eta)\in\H$ such that
$$
\varphi,\psi\in H^2({\mathfrak I}),\quad
\tilde\varphi,\tilde\psi\in H_0^1({\mathfrak I}),\quad
\eta \in \D(T),\quad
\int_0^\infty \mu(s)\eta(s)\, \d s \in H^2({\mathfrak I}).
$$
According to \cite{MRFS},
the operator $\A$ is the infinitesimal generator of
a contraction semigroup
$$
S(t) = \e^{t\A}:\H\to\H.
$$
In particular, $\A$ is dissipative. Indeed, for every $z=(\varphi,\tilde\varphi,\psi,\tilde\psi,\eta)\in\D(A)$,
\begin{equation}
\label{TTT}
\l \A z,z\r_{\H}=\l T \eta,\eta\r_{\M}
=\frac12\int_0^{\infty}
\mu'(s)\|\eta_x(s)\|^2\,\d s\leq 0.
\end{equation}
Thus, for every initial datum
$z=(\varphi_0,\tilde\varphi_0,\psi_0,\tilde\psi_0,\eta_0)\in\H$
given at time $t=0$,
the unique solution at time $t>0$ to \eqref{EQLIN} reads
$$Z(t)=(\varphi(t),\varphi_t(t),\psi(t),\psi_t(t),\eta^t)=S(t)z,$$
where $\eta^t$ fulfills the explicit representation
formula (see~\cite{Terreni})
\begin{equation}
\label{REP}
\eta^t(s)=
\begin{cases}
\psi(t)-\psi(t-s) &  s\leq t,\\
\eta_0(s-t)+\psi(t)-\psi_0 & s>t.
\end{cases}
\end{equation}

\section{Theorem \ref{MAIN} (Necessity)}
\label{Nec}

\noindent
The proof of the necessity part of Theorem \ref{MAIN} is essentially the same
of~\cite[Theorem~3.2]{ChePat}, dealing with a linearly viscoelastic
equation. For the reader's convenience, we report here the short argument.

\subsection*{Proof of Theorem \ref{MAIN} (Necessity)}
Suppose $S(t)$ exponentially stable on $\H$. Then, for any initial datum $z\in \H$
of the form $z=(0,0,0,0,\eta_0)$ we have
$$\max\{\|\psi_x(t)\|,\|\eta^t\|_\M\}
\textstyle \leq \max\{1/a,1\}\|S(t)z\|_{\H}\leq K \e^{-\omega t}\|\eta_0\|_\M$$
for some positive $K,\omega$.
On the other hand, exploiting the representation formula \eqref{REP},
\begin{align*}
2\|\eta^t\|_\M^2
&\geq2\int_t^\infty\mu(s)\|\eta_{0x}(s-t)+\psi_x(t)\|^2\,\d s\\
&\geq\int_t^\infty \mu(s)\|\eta_{0x}(s-t)\|^2\,\d s
-2m\|\psi_x(t)\|^2\\
\noalign{\vskip2mm}
&\geq \|\Sigma(t)\eta_0\|^2_\M-2m K^2 \e^{-2\omega t}\|\eta_0\|_\M^2.
\end{align*}
We conclude that
$$\|\Sigma(t)\eta_0\|_{\M}\leq
K\sqrt{2(1+m)}\,\e^{-\omega t}\|\eta_0\|_{\M},
$$
and the claim is a consequence of Theorem \ref{THM33}.
\qed

\medskip

Our Theorem \ref{MAIN} (Necessity), seems to contradict
the following result established by Messaoudi and Said-Houari:

\begin{theorem}[Theorem 2.1 in \cite{MS}]
\label{sterco}
Let $\mu$ satisfy the differential inequality
\begin{equation}
\label{sterco2}
\mu'(s)+\delta [\mu(s)]^p\leq 0
\end{equation}
for some $\delta>0$ and some $p\in(1,\frac32)$.
Then, for every initial datum $z\in\H$, the inequality
$$E(t):={\tx\frac12}\|S(t)z\|_\H^2\leq \frac{M}{(1+t)^{1/(p-1)}}
$$
holds for some $M>0$ depending on $z$.
\end{theorem}

Let us observe that
the correct conclusion of Theorem~\ref{sterco} should have been that the energy decays
exponentially, for lack of exponential stability prevents the existence of uniform decay
patterns. Indeed, the thesis of Theorem~\ref{sterco} implies that, for every $z\in\H$,
$$\|(1+t)^{1/(2p-2)}S(t)z\|_\H\leq Q_z
$$
for some $Q_z>0$, and a direct application of the
Uniform Boundedness Principle yields
$$\|S(t)z\|_\H\leq \frac{Q}{(1+t)^{1/(2p-2)}}\|z\|_\H,
$$
where $Q>0$ is now independent of $z$.
Hence the operator norm of $S(t)$ goes below one
for large values of $t$, and exponential stability
readily follows.
At the same time, $S(t)$ cannot have a uniform
decay
if, for instance,
$$\mu(s)=\frac{1}{(1+s)^{1/(p-1)}},$$
which complies with \eqref{sterco2} but clearly violates~\eqref{NEC}.

\section{Theorem \ref{MAIN} (Sufficiency)}
\label{Suf}

\noindent
The exponential stability of $S(t)$ has been proved in \cite{MRFS} within
the hypotheses
\begin{equation}
\label{riv1}
\mu'(s)+k_1\mu(s)\geq 0,\qquad |\mu''(s)|\leq k_2\mu(s),
\end{equation}
for some $k_1,k_2>0$, and
\begin{equation}
\label{riv2}
\mu'(s)+\delta\mu(s)\leq 0,
\end{equation}
for some $\delta>0$.
Let aside \eqref{riv1}, which is only technical, condition~\eqref{riv2}
is equivalent to~\eqref{NEC} with $C=1$.
Nonetheless, \eqref{NEC} with $C>1$
turns out to be much more general than \eqref{riv2}.
For instance, any compactly supported $\mu$ (in the class of kernels considered in the present paper)
satisfies \eqref{NEC}, but cannot comply with \eqref{riv2} if it has flat zones, or even
horizontal inflection points. Besides, \eqref{NEC} with $C>1$ makes no assumptions at all
on the derivative $\mu'$.

\smallskip
Analogously to \cite{MRFS},
the proof of the sufficiency part of Theorem \ref{MAIN}
is based on the following abstract result from~\cite{Pru}
(see also~\cite{GNP} for the precise statement used here).

\begin{lemma}
\label{pruss}
The contraction semigroup $S(t)$ on $\H$ is exponentially stable
if and only if there exists $\eps>0$ such that
$$
\inf_{\lambda\in \R}\|\ii\lambda z-\A z\|_{\H}\geq \eps\|z\|_{\H},\quad \forall z \in \D(\A),
$$
where $\A$ and $\H$ are understood to be the complexifications
of the original $\A$ and $\H$.
\end{lemma}

\subsection*{Proof of Theorem \ref{MAIN} (Sufficiency)}
Within hypothesis~\eqref{NEC}, suppose $S(t)$ be not exponentially stable.
Then, Lemma~\ref{pruss} ensures the existence of
sequences $\lambda_n\in\R$ and
$z_n=(\varphi_n,\tilde\varphi_n,\psi_n,\tilde\psi_n,\eta_n)\in \D(\A)$ with
$$
\|z_n\|_\H^2=\kappa\|\varphi_{nx}+\psi_n\|^2 +\rho_1\|\tilde\varphi_n\|^2
+a\|\psi_{nx}\|^2+\rho_2\|\tilde\psi_n\|^2 + \|\eta_n\|^2_\M=1
$$
satisfying the relation
\begin{equation}
\label{ZERO}
\ii\lambda_n z_n-\A z_n\to 0\quad\text{in } \H.
\end{equation}
Componentwise, \eqref{ZERO} reads
\begin{align}
\label{quattro}
&\ii\lambda_n \varphi_n- \tilde\varphi_n\to 0\quad\text{in } H_0^1,\\
\noalign{\vskip1.5mm}
\label{cinque}
&\ii\lambda_n \rho_1\tilde\varphi_n- \kappa(\varphi_{nx}+\psi_n)_x\to 0\quad\text{in } L^2,\\
\noalign{\vskip1.5mm}
\label{UNO}
&\ii\lambda_n \psi_n- \tilde\psi_n\to 0\quad\text{in } H_0^1,\\
\label{DUE}
&\ii\lambda_n\rho_2 \tilde\psi_n -a\psi_{nxx}
-\int_0^\infty\mu(s)\eta_{nxx}(s)\,\d s+\kappa(\varphi_{nx}+\psi_n)\to 0\quad\text{in } L^2,\\
\label{TRE}
&\ii\lambda_n {\eta}_n-T \eta_n-\tilde\psi_n\to 0\quad\text{in } \M.
\end{align}

\smallskip
\noindent
We assume $\lambda_n\not\to 0$ (the case $\lambda_n\to 0$ is much simpler and left to the reader).
Accordingly,
up to a subsequence ({\uts} in the sequel),
$$
\lambda_n\to \lambda_\star\in[-\infty,\infty]\setminus\{0\}.
$$
We will reach a contradiction by showing that every single component of $z_n$ goes to zero
in its norm \uts.\footnote{It is understood that passing to a subsequence means to refine the former one.}
The first part of the proof borrows some ideas from~\cite{Flat}.
Since $\mu$ is an absolutely continuous function vanishing at infinity,  the set
$$\II=\big\{s\in\R^+:\, K\mu'(s)+\mu(s)<0\big\}
$$
has positive measure for some $K>0$ large enough.
Let us define the space
$${\mathcal S}=L^2_\mu(\II;H_0^1({\mathfrak I})).
$$
We need three preliminary lemmas. We will lean
several times (without explicit mention) on the boundedness in $\H$ of $z_n$.

\begin{lemma}
\label{lemmaprim}
We have the convergence
$\|\eta_n\|_{\mathcal S}\to 0$.
\end{lemma}

\begin{proof}
By means of \eqref{ZERO},
$$\Re\l \ii\lambda_n z_n-\A z_n, z_n\r_\H
=-\Re\l \A z_n, z_n\r_\H\to 0,$$
and using~\eqref{TTT} we are led to
\begin{equation}
\label{MUprimo}
0\leq-\int_\II\mu'(s)\|\eta_{nx}(s)\|^2\,\d s\leq
-\int_0^\infty\mu'(s)\|\eta_{nx}(s)\|^2\,\d s\to 0.
\end{equation}
Since
$$
\|\eta_n\|_{\mathcal S}^2=\int_\II\mu(s)\|\eta_{nx}(s)\|^2\,\d s
\leq -K\int_\II\mu'(s)\|\eta_{nx}(s)\|^2\,\d s,
$$
the conclusion follows.
\end{proof}

\begin{lemma}
\label{DUDU}
The sequence
$|\lambda_n|\|\tilde\psi_n\|_{*}$ is bounded,
where $\|\cdot\|_{*}$  is the norm in $H^{-1}(\mathfrak{I})$.
\end{lemma}

\begin{proof}
By the triangle inequality,
\begin{align*}
\rho_2|\lambda_n|\|\tilde\psi_n\|_{*}
&\leq\Big\|\ii\lambda_n\rho_2 \tilde\psi_n -
a\psi_{nxx} -\int_0^\infty\mu(s)\eta_{nxx}(s)\,\d s+\kappa(\varphi_{nx}+\psi_n)\Big\|_{*}\\
&\quad +\Big\|a\psi_{nxx}+
\int_0^\infty\mu(s)\eta_{nxx}(s)\,\d s-\kappa(\varphi_{nx}+\psi_n)\Big\|_{*}.
\end{align*}
Due to \eqref{DUE} and the continuous embedding
$L^2(\mathfrak{I})\subset  H^{-1}(\mathfrak{I})$, the first term in the right-hand side
goes to zero, whereas the second one is dominated by
$$
a\|\psi_{nx}\|+\int_0^\infty\mu(s)\|\eta_{nx}(s)\|\,\d s+\kappa\|\varphi_n\| +\kappa\|\psi_n\|_{*},
$$
which is bounded uniformly with respect to $n\in\N$.
\end{proof}

\begin{lemma}
\label{lemmaTEO}
Within \eqref{NEC}, for any $\xi\in\M$ we have the estimate
$$\int_0^\infty \mu(s)\int_0^s
\|\xi_x(\sigma)\|\,\d \sigma\d s\leq
\tx \frac{\sqrt{4C m}}\delta\|\xi\|_\M.
$$
\end{lemma}

\begin{proof}
Exploiting \eqref{NEC} and the H\"older inequality,
$$\int_0^\infty \mu(s)\int_0^s
\|\xi_x(\sigma)\|\,\d \sigma\d s\leq
\sqrt{C}\int_0^\infty \sqrt{\mu(s)}\,G(s)\leq \sqrt{C m}\|G\|_{L^2(\R^+)},
$$
having set
$$G(s)=\int_0^s \e^{-\frac\delta2(s-\sigma)}\sqrt{\mu(\sigma)}\,\|\xi_x(\sigma)\|\,\d \sigma\d s.
$$
By a well-known result of measure theory,
$$\tx \|G\|_{L^2(\R^+)}\leq\frac2\delta\big\|\sqrt{\mu}\,\|\xi_x\|\big\|_{L^2(\R^+)}=\frac2\delta\|\xi\|_\M,
$$
which completes the argument.
\end{proof}

We are now ready to prove the norm-decay of every single component of $z_n$.

\begin{lemma}
\label{lemmav}
The convergence $\|\tilde\psi_n\|\to 0$ holds {\uts}.
\end{lemma}

\begin{proof}
Introducing $\Psi_n$ such that $-\Psi_{nxx}=\tilde\psi_n$, we infer from
Lemmas~\ref{lemmaprim} and \ref{DUDU} that
$$|\ii\lambda_n\l\eta_n,\Psi_n\r_{\mathcal S}|\leq
|\lambda_n|\|\tilde\psi_n\|_{*}\int_\II\mu(s)\|\eta_{nx}(s)\|\,\d s
\leq \sqrt{m}\,|\lambda_n|\|\tilde\psi_n\|_{*}\|\eta_n\|_{\mathcal S}\to 0.
$$
On the other hand, calling
$$
\xi_n=\ii\lambda_n {\eta}_n-T \eta_n-\tilde\psi_n,
$$
we find the explicit expression
\begin{equation}
\label{save}
\eta_n(s)=\frac{1}{\ii\lambda_n}(1-\e^{-\ii\lambda_n s})\tilde\psi_n
+\int_0^s \e^{-\ii\lambda_n (s-\sigma)}\xi_n(\sigma)\,\d \sigma.
\end{equation}
Hence,
$$
\ii\lambda_n\l\eta_n,\Psi_n\r_{\mathcal S}=\alpha_n\|\tilde\psi_n\|^2
+\beta_n\to 0,
$$
where we put
\begin{align*}
\alpha_n&=\int_\II \mu(s)(1-\e^{-\ii\lambda_n s})\,\d s,\\
\beta_n&=\ii\lambda_n\int_\II \mu(s)\int_0^s \e^{-\ii\lambda_n (s-\sigma)}
\l\xi_{nx}(\sigma),\Psi_{nx}\r\,\d \sigma\d s.
\end{align*}
The conclusion follows by showing that $\beta_n\to 0$
whereas $\alpha_n$ remains away from zero for large $n$.
Indeed, by Lemmas \ref{DUDU}-\ref{lemmaTEO}
and the convergence~\eqref{TRE}
we get
$$
|\beta_n|\leq \tx \frac{\sqrt{4Cm}}{\delta}|\lambda_n|\|\tilde\psi_n\|_{*}\|\xi_n\|_\M\to 0.
$$
Concerning $\alpha_n$, we have two possibilities.
If $\lambda_\star\in\{-\infty,\infty\}$,
the Riemann-Lebesgue lemma
yields the convergence
$$\alpha_n\to
\int_\II \mu(s)\,\d s>0,$$
whereas, if $\lambda_\star\in\R\setminus\{0\}$,
$$\Re \alpha_n\to \int_\II \mu(s)(1-\cos \lambda_\star s)\,\d s>0.$$
In either case $\Re \alpha_n$ has positive limit (again, {\uts}).
\end{proof}

\begin{lemma}
\label{lemmaduo}
The convergence $\|\psi_{nx}\|\to 0$ holds {\uts}.
\end{lemma}

\begin{proof}
Define
$$\zeta_n(s)=\frac{1}{\ii\lambda_n}
(1-\e^{-\ii\lambda_n s})(\tilde\psi_n-\ii\lambda_n \psi_n).
$$
By means of \eqref{UNO},
it is apparent that
$\zeta_n\to 0$ in $\M$ and, thanks to \eqref{save},
\begin{equation}
\label{save2}
\eta_n(s)=
(1-\e^{-\ii\lambda_n s})\psi_n
+\int_0^s \e^{-\ii\lambda_n (s-\sigma)}\xi_n(\sigma)\,\d\sigma
+\zeta_n(s),
\end{equation}
which, on account of~Lemma~\ref{lemmaprim}, entails
$$
\l\eta_n,\psi_n\r_{\mathcal S}-\l\zeta_n,\psi_n\r_{\mathcal S}
=\alpha_n\|\psi_{nx}\|^2
+\gamma_n\to 0,
$$
with $\alpha_n$ as above and
$$\gamma_n=\int_\II \mu(s)\int_0^s \e^{-\ii\lambda_n (s-\sigma)}
\l\xi_{nx}(\sigma),\psi_{nx}\r\,\d\sigma\d s.
$$
An application of Lemma \ref{lemmaTEO} gives
$$
|\gamma_n|\leq
\tx\frac{\sqrt{4mC}}{\delta}\|\psi_{nx}\|\|\xi_n\|_\M\to 0.
$$
Knowing that $\alpha_n$ remains away from zero,
we conclude that $\|\psi_{nx}\|\to 0$.
\end{proof}

\begin{lemma}
\label{lemmatrio}
The convergence $\|\eta_n\|_\M\to 0$ holds \uts.
\end{lemma}

\begin{proof}
Making use of \eqref{save2},
we easily obtain the estimate
$$\|\eta_n\|_\M^2\leq 2\sqrt{m}\,\|\psi_{nx}\|\|\eta_n\|_\M
+\int_0^\infty \mu(s)\|\eta_{nx}(s)\|\int_0^s\|\xi_{nx}(\sigma)\|\,\d \sigma\d s
+\|\zeta_n\|_\M\|\eta_n\|_\M.
$$
Arguing exactly as in Lemma \ref{lemmaTEO}, we see that
$$\int_0^\infty \mu(s)\|\eta_{nx}(s)\|\int_0^s\|\xi_{nx}(\sigma)\|\,\d \sigma\d s
\leq \tx \frac{\sqrt{4C}}\delta\|\eta_n\|_\M\|\xi_n\|_\M.
$$
Consequently,
$$\|\eta_n\|_\M\leq 2\sqrt{m}\,\|\psi_{nx}\|
+\tx \frac{\sqrt{4C}}\delta\|\xi_n\|_\M
+\|\zeta_n\|_\M\to 0
$$
on account of Lemma~\ref{lemmaduo}.
\end{proof}

At this point, we introduce the sequence of functions
$$
F_n(x)= a\psi_{nx}(x) + \int_0^\infty \mu(s)\eta_{nx}(x,s)\, \d s.
$$
An asymptotic control of certain boundary terms will be needed.

\begin{lemma}
\label{bordo}
The convergence
$F_n(x)\overline{\varphi_{nx}(x)}\to 0$ holds {\uts} for $x=0$ and $x=\ell$.
\end{lemma}

\begin{proof}
The argument follows the lines of \cite{MRFS}. Accordingly,
choose a real function $q\in\C^1(\mathfrak{I})$ satisfying
$q(0)=-q(\ell)=1$. A multiplication of \eqref{DUE} by
$$
Q_n(x)= q(x) F_n(x)
$$
yields the convergence
$$
\ii\lambda_n\rho_2 \l \tilde\psi_n,Q_n \r-\l F_{nx},Q_n\r + \kappa \l \varphi_{nx}+\psi_n,Q_n \r \to 0.
$$
By Lemmas \ref{lemmaduo} and \ref{lemmatrio},
it is readily seen that $F_n\to0$ in $L^2(\mathfrak{I})$, thus
$$
\kappa \l \varphi_{nx}+\psi_n,Q_n \r \to 0.
$$
We claim that
$$
\Re \ii\lambda_n \l \tilde\psi_n,Q_n \r \to 0.
$$
Indeed, exploiting \eqref{UNO} we get
$$
\ii\lambda_n\l q \tilde\psi_n, a \psi_{nx} \r
= -a \l q \tilde\psi_n, \ii\lambda_n \psi_{nx} \r
= -a \l q \tilde\psi_n, \tilde\psi_{nx} \r + \eps_n
$$
for some complex sequence $\eps_n\to0$, whereas \eqref{TRE} and an integration by parts yield
(the boundary terms disappear as in \cite{Terreni})
\begin{align*}
\ii\lambda_n\int_0^\infty \mu(s) \l q \tilde\psi_n,\eta_{nx}(s)\r \, \d s
&=-\int_0^\infty \mu(s) \l q \tilde\psi_n,\ii\lambda_n\eta_{nx}(s)\r\, \d s  \\
&= - \int_0^\infty \mu'(s)
\l q\tilde\psi_n,\eta_{nx}(s) \r \,\d s -m \l q \tilde\psi_n, \tilde\psi_{nx} \r + \nu_n
\end{align*}
for some complex sequence $\nu_n\to0$.
Collecting the two identities, we obtain
$$
\ii\lambda_n\l\tilde\psi_n, Q_n \r=
-b \l q \tilde\psi_n, \tilde\psi_{nx} \r - \int_0^\infty \mu'(s)
\l q\tilde\psi_n,\eta_{nx}(s) \r \,\d s + \eps_n + \nu_n.
$$
Using Lemma \ref{lemmav} and integrating by parts, we infer that
$$
2\Re \l q \tilde\psi_n, \tilde\psi_{nx}\r =\int_0^\ell q(x) {\frac{\d}{\d x}}|\tilde\psi_n(x)|^2\, \d x =
-\l q' \tilde\psi_n, \tilde\psi_n\r \to 0.
$$
Besides, invoking \eqref{MUprimo},
$$
\int_0^\infty \mu'(s)  \l q\tilde\psi_n,\eta_{nx}(s) \r \,\d s\to 0,
$$
and the claim is established. Summarizing,
we arrive at the convergence
$$
\Re \l F_{nx}, Q_n \r \to 0.
$$
Writing
$$
-2\Re \l F_{nx}, Q_n \r=
|F_n(\ell)|^2 + |F_n(0)|^2+\l q' F_n, F_n \r,
$$
and noting that the last term in the right-hand side vanishes as $n\to\infty$,
we reach the conclusion
$$
|F_n(\ell)|^2 + |F_n(0)|^2 \to 0.
$$
To finish the proof it is now enough showing that the sequence
$$
|\varphi_{nx}(\ell)|^2 + |\varphi_{nx}(0)|^2
= - 2 \Re \l q\varphi_{nx},\varphi_{nxx}\r - \l q'\varphi_{nx},\varphi_{nx} \r
$$
is bounded. So is certainly the second term in the right-hand side
above. Concerning the first one,
multiplying~\eqref{cinque} by $q \varphi_{nx}$,
and taking advantage of~\eqref{quattro} and Lemma~\ref{lemmaduo}, we have
$$
-\rho_1 \l q' \tilde\varphi_n,\tilde\varphi_n\r + 2\kappa\,
\Re \l q\varphi_{nx},\varphi_{nxx}\r
=2\rho_1 \Re \l q\tilde\varphi_n,\tilde\varphi_{nx}\r + 2\kappa\,
\Re \l q\varphi_{nx},\varphi_{nxx}\r \to 0.
$$
Since $\l q' \tilde\varphi_n,\tilde\varphi_n\r$ is
bounded, the same is true for $\Re \l q\varphi_{nx},\varphi_{nxx}\r$.
\end{proof}

\begin{remark}
Observe that, within the Neumann boundary condition for either $\varphi$ or $\psi$,
Lemma~\ref{bordo} is trivially true.
\end{remark}

\begin{lemma}
\label{lemmaMisto}
Within condition \eqref{rivera}, the convergence
$\|\varphi_{nx}+\psi_n\|\to 0$ holds \uts.
\end{lemma}

\begin{proof}
Multiplying \eqref{DUE} by $\varphi_{nx}+\psi_n$ and exploiting \eqref{cinque}, we obtain
$$
\kappa\|\varphi_{nx}+\psi_n\|^2 -F_n \overline{\varphi_{nx}} \big|_{0}^{\ell}
+ \ii\lambda_n\rho_2 \l \tilde\psi_n,\varphi_{nx}+\psi_n\r
- \ii\lambda_n\frac{\rho_1}{\kappa} \l F_n, \tilde\varphi_n\r
\to 0.
$$
The boundary term goes to zero by Lemma~\ref{bordo}. Besides,
multiplying \eqref{UNO} by $\tilde\psi_n$ and applying Lemma \ref{lemmav}, we learn that
$$
\ii\lambda_n\rho_2 \l \tilde\psi_n,\psi_n\r\to 0,
$$
and we deduce the convergence
$$
\kappa\|\varphi_{nx}+\psi_n\|^2+ \ii\lambda_n\rho_2 \l \tilde\psi_n,\varphi_{nx}\r
- \ii\lambda_n\frac{\rho_1}{\kappa} \l F_n, \tilde\varphi_n\r
\to 0.
$$
Finally, using \eqref{rivera}, \eqref{quattro}, \eqref{UNO} and \eqref{TRE}, we have
\begin{align*}
&\ii\lambda_n\rho_2 \l \tilde\psi_n,\varphi_{nx}\r- \ii\lambda_n\frac{\rho_1}{\kappa} \l F_n, \tilde\varphi_n\r\\
&=-\rho_2 \l \tilde\psi_n,\ii\lambda_n\varphi_{nx}\r
-\frac{b\rho_1}{\kappa} \l \ii\lambda_n\psi_{nx}, \tilde\varphi_n\r
-\frac{\rho_1}{\kappa}\int_0^\infty \mu(s)
\l \ii\lambda_n \eta_{nx}(s)-\ii\lambda_n\psi_{nx},\tilde\varphi_n\r\, \d s\\
&= b\Big(\frac{\rho_1 }{\kappa} -\frac{\rho_2}{b} \Big)\l \tilde\psi_n,\tilde\varphi_{nx}\r
-\frac{\rho_1}{\kappa}\int_0^\infty \mu'(s) \l \eta_{nx}(s),\tilde\varphi_n\r\, \d s + \eps_n\\
&= -\frac{\rho_1}{\kappa}\int_0^\infty \mu'(s) \l \eta_{nx}(s),\tilde\varphi_n\r\, \d s + \eps_n,
\end{align*}
for some complex sequence $\eps_n\to 0$.
Since \eqref{MUprimo} bears
$$
\int_0^\infty \mu'(s) \l \eta_{nx}(s),\tilde\varphi_n\r\, \d s\to 0,
$$
the proof is finished.
\end{proof}

On account of Lemmas \ref{lemmav}, \ref{lemmaduo}, \ref{lemmatrio} and \ref{lemmaMisto},
the sought contradiction is attained once we prove the convergence
$\|\tilde\varphi_n\|\to 0$ \uts.
To this end, a multiplication of \eqref{cinque} by $\varphi_n$ will do.
\qed




\begin{thebibliography}{99}


\bibitem{AMM}
{\au F. Ammar Khodja, A. Benabdallah, J.E. Muñoz Rivera and R. Racke},
{\ti Energy decay for Timoshenko systems of memory type},
{\jou J.\ Differential Equations}
\no{194}{82--115}{2003}

\bibitem{ChePat}
{\au V.V. Chepyzhov and V. Pata},
{\ti Some remarks on stability of semigroups arising from linear viscoelasticity},
{\jou Asymptot.\ Anal.}
\no{46}{251--273}{2006}

\bibitem{DAF}
{\au C.M. Dafermos},
{\ti Asymptotic stability in viscoelasticity},
{\jou Arch.\ Rational Mech.\ Anal.}
\no{37}{297--308}{1970}

\bibitem{GNP}
{\au C. Giorgi, M.G. Naso and V. Pata},
{\ti Exponential stability in linear heat conduction with memory: a semigroup approach},
{\jou Comm.\ Appl.\ Anal.}
\no{5}{121--134}{2001}

\bibitem{Terreni}
{\au M. Grasselli and V. Pata},
{\ti Uniform attractors of nonautonomous systems with memory},
in ``Evolution Equations, Semigroups and Functional Analysis''
(A.~Lorenzi and B.~Ruf, Eds.),
\eds{pp.155--178, Progr.\ Nonlinear Differential Equations
Appl.\ no.50, Birkh\"{a}user}{Boston}{2002}

\bibitem{MS}
{\au S.A. Messaoudi and B. Said-Houari},
{\ti Uniform decay in a Timoshenko-type system with past memory},
{\jou J.\ Math.\ Anal.\ Appl.}
\no{360}{459--475}{2009}

\bibitem{MRFS}
{\au J.E. Mu{\~n}oz Rivera and H.D. Fern\'andez Sare},
{\ti Stability of Timoshenko systems with past history},
{\jou J.\ Math.\ Anal.\ Appl.}
\no{339}{482--502}{2008}

\bibitem{RR}
{\au J.E. Mu{\~n}oz Rivera and R. Racke},
{\ti Mildly dissipative nonlinear Timoshenko systems--global existence and exponential stability},
{\jou J.\ Math.\ Anal.\ Appl.}
\no{276}{248--278}{2002}

\bibitem{Flat}
{\au V. Pata},
{\ti Exponential stability in linear viscoelasticity with almost flat memory kernels},
{\jou Commun.\ Pure\ Appl.\ Anal.}
\no{9}{721--730}{2010}

\bibitem{Pru}
{\au J. Pr\"{u}ss},
{\ti On the spectrum of $\text{C}_0$-semigroups},
{\jou Trans.\ Amer.\ Math.\ Soc.}
\no{284}{847--857}{1984}

\bibitem{SOU}
{\au A. Soufyane},
{\ti Stabilisation de la poutre de Timoshenko},
{\jou C.\ R.\ Acad.\ Sci.\ Paris S\'er.\ I Math.}
\no{228}{731--734}{1999}

\bibitem{TIM}
{\au S.P. Timoshenko},
{\ti On the correction for shear of a differential equation for transverse vibrations of prismatic bars},
{\jou Philos.\ Magazine}
\no{41}{744--746}{1921}

\end{thebibliography}
\end{document}